\documentclass[preprint]{elsarticle}

\usepackage{amssymb,amsmath}
\usepackage{amsthm}
\usepackage{ifthen}
\usepackage{tikz}
\usepackage{framed}
\usepackage{framed}
\newtheorem{theorem}{Theorem}[section]
\newif\ifskip
\skiptrue
\newtheorem{corollary}[theorem]{Corollary}
\newtheorem{proposition}[theorem]{Proposition}
\newtheorem{definition}{Definition}
\newtheorem{definitions}[definition]{Definitions}
\newtheorem{observation}[theorem]{Observation}
\newtheorem{example}{Example}[section]
\newtheorem{remark}[theorem]{Remark}
\newtheorem{Problem}{Problem}
\newtheorem{problem}[Problem]{Problem}
\newtheorem{lemma}[theorem]{Lemma}
\ifskip\else

\theoremstyle{plain}

\newtheorem{Definition}{Definition}

\newtheorem{Problem}[Definition]{Problem}

\theoremstyle{break}

\newtheorem{definition}[Definition]{Definition}

\newtheorem{lemma}[Definition]{Lemma}
\newtheorem{theorem}[Definition]{Theorem}
\newtheorem{proposition}[Definition]{Proposition}

\newtheorem{remark}[Definition]{Remark}

\newtheorem{corollary}[Definition]{Corollary}

\newtheorem{example}[Definition]{Example}

\newtheorem{problem}[Definition]{Problem}

\newtheorem{observation}[Definition]{Observation}
\fi %skip
\newif\ifskip
\skiptrue
\newif\ifrevised
\revisedtrue
\begin{document}
\begin{frontmatter}
\title{On Sequences of Polynomials \\ Arising from Graph Invariants
%\tnoteref{t1,t2}
}
%--------------------authors------------------------------

\author[tk]{T.~Kotek\fnref{fn0}}
\ead{kotek@forsyte.at}
\ead[url]{http://forsyte.at/people/kotek}

\author[jam]{J.A.~Makowsky\corref{cor1}\fnref{fn1}}
\ead{janos@cs.technion.ac.il}
\ead[url]{http://www.cs.technion.ac.il/~janos}

\author[evr]{E.V.~Ravve\fnref{fn2}}
\ead{cselena@braude.ac.il}
%\ead[url]{xxxx}
%---
\cortext[cor1]{Corresponding author}
\fntext[fn0]{Work done in part while the author was visiting the Simons Institute
for the Theory of Computing in Fall 2016.}
\fntext[fn1]{Work done in part while the author was visiting the Simons Institute
for the Theory of Computing in Spring and Fall 2016.}
\fntext[fn2]{Visiting scholar at the Faculty of Computer Science, Technion--IIT, Haifa, Israel}
%---------
\address[tk]{
Institut f\"ur Informationssysteme, Technische Universit\"at Wien Vienna, Austria}
\address[jam]{Department of Computer Science, Technion--IIT, Haifa, Israel}
\address[evr]{Department of Software Engineering, ORT-Braude College, Karmiel, Israel}
%------------------------------------------------
%\begin{document}
\begin{abstract}
Graph polynomials are deemed useful if they give rise to algebraic characterizations of various graph properties,
and their evaluations encode many other graph invariants. 
Algebraic:
The complete graphs $K_n$ and the complete bipartite graphs $K_{n,n}$  
can be characterized as those graphs whose matching polynomials
satisfy a certain recurrence relations and are related to the Hermite and Laguerre polynomials.
An encoded graph invariant: The absolute value of the 
chromatic polynomial $\chi(G,X)$ of a graph $G$ evaluated at $-1$ counts the number of acyclic orientations of $G$.

In this paper we prove a general theorem on graph families
which are characterized by families of polynomials satisfying linear recurrence relations.
This gives infinitely many instances similar to the characterization of $K_{n,n}$.
We also show where to use, instead of
the Hermite and Laguerre polynomials, 
linear recurrence relations
where the coefficients do not depend on $n$.

Finally, we discuss the distinctive power of graph polynomials in specific form.

\end{abstract}
%\end{keyworkd}

%\maketitle
\end{frontmatter}
\newpage
\tableofcontents
\newpage
%------------------------------------------------
\newif\ifsubmit
%\submitfalse
\submittrue
%MATH MODE ABBREVIATIONS
\newcommand{\WFF}{\mathbf{WFF}}
\newcommand{\MSOL}{\mathbf{MSOL}}
\newcommand{\SOL}{\mathbf{SOL}}
\newcommand{\SOLEVAL}{\mathbf{SOLEVAL}}
\newcommand{\CMSOL}{\mathbf{CMSOL}}
\newcommand{\SEN}{\mbox{\bf SEN}}
\newcommand{\WFTF}{\mbox{\bf WFTF}}
\newcommand{\FOL}{\mbox{\bf FOL}}
\newcommand{\TFOF}{\mbox{\bf TFOF}}
\newcommand{\TFOL}{\mbox{\bf TFOL}}
\newcommand{\FOF}{\mbox{\bf FOF}}
\newcommand{\NNF}{\mbox{\bf NNF}}
\newcommand{\N}{{\mathbb N}}
\newcommand{\bN}{{\mathbb N}}
\newcommand{\bR}{{\mathbb R}}
\newcommand{\HF}{\mbox{\bf HF}}
\newcommand{\CNF}{\mbox{\bf CNF}}
\newcommand{\PNF}{\mbox{\bf PNF}}
\newcommand{\QF}{\mbox{\bf QF}}
\newcommand{\DNF}{\mbox{\bf DNF}}
\newcommand{\DISJ}{\mbox{\bf DISJ}}
\newcommand{\CONJ}{\mbox{\bf CONJ}}
\newcommand{\Ass}{\mbox{Ass}}
\newcommand{\Var}{\mbox{Var}}
\newcommand{\Support}{\mbox{Support}}
\newcommand{\V}{\mbox{\bf Var}}
\newcommand{\fA}{{\mathfrak A}}
\newcommand{\fB}{{\mathfrak B}}
\newcommand{\fN}{{\mathfrak N}}
\newcommand{\fZ}{{\mathfrak Z}}
\newcommand{\fQ}{{\mathfrak Q}}
\newcommand{\Aa}{{\mathfrak A}}
\newcommand{\Bb}{{\mathfrak B}}
\newcommand{\Cc}{{\mathfrak C}}
\newcommand{\Gg}{{\mathfrak G}}
\newcommand{\Ww}{{\mathfrak W}}
\newcommand{\Rr}{{\mathfrak R}}
\newcommand{\Nn}{{\mathfrak N}}
\newcommand{\Zz}{{\mathfrak Z}}
\newcommand{\Qq}{{\mathfrak Q}}
\newcommand{\F}{{\mathbf F}}
\newcommand{\T}{{\mathbf T}}
\newcommand{\Z}{{\mathbb Z}}
\newcommand{\R}{{\mathbb R}}
\newcommand{\C}{{\mathbb C}}
\newcommand{\Q}{{\mathbb Q}}
\newcommand{\cC}{{\mathcal C}}
\newcommand{\cD}{{\mathcal D}}
\newcommand{\cG}{{\mathcal G}}
\newcommand{\cR}{{\mathcal R}}
\newcommand{\cS}{{\mathcal S}}
\newcommand{\cY}{{\mathcal Y}}
\newcommand{\bP}{{\mathbf P}}
\newcommand{\bQ}{{\mathbf Q}}
\newcommand{\bNP}{{\mathbf{NP}}}
\newcommand{\MT}{\mbox{MT}}
\newcommand{\TT}{\mbox{TT}}
%---------------------------------------
\newcommand{\bX}{\bar{X}}
\newcommand{\bY}{\bar{Y}}
\newcommand{\FPT}{{\mathbf FPT}}
\newcommand{\PTime}{{\mathbf PTime}}

\newcommand{\DU}{\mathrm{DU}}

%----------------------------------------------------------------
\ifsubmit\else
\begin{framed}
\noindent
File: NREV-intro
\\
Label:  se:intro
\end{framed}
\fi %submit
\ \hfill In memoriam Herbert Wilf, June 13, 1931 -  January 7, 2012
%-------------------------------------------------------------
\section{Introduction and background}
\label{se:intro}
%-------------------------------------------------------------
\subsection{Wilf's Recognition problem}
H. Wilf asked in \cite{ar:Wilf1973} to characterize and recognize the instances of the chromatic polynomial.
C.D. Godsil and I. Gutman \cite{ar:GodsilGutman1981} gave a characterization of the
instances of the defect matching polynomial $\mu(G;X)$ for paths $P_n$, cycles $C_n$, complete graphs $K_n$
and bipartite complete graphs $K_{n,n}$ in terms of orthogonal polynomials.
We want to put Wilf's question and C.D. Godsil and I. Gutman's observation into a larger perspective.
First we have to fix some terminology.
Let $\cG$ denote the class of all finite graphs with no multiple edges.
A graph property is a class of graphs $\cC \subseteq \cG$ closed under graph isomorphism.
A graph parameter $f(G)$ is a function $\cG \rightarrow \Z$ invariant under graph isomorphism.
A graph polynomial with $r$ indeterminates\footnote{If the polynomial is univariate, we write $X$ instead of $\bX$}
$\bX = (X_1, \ldots , X_r)$ 
is a function $\bP$ from all finite graphs
into the polynomial ring $\Z[\bX]$ which is invariant under graph isomorphism.
We write $\bP(G; \bX)$ for the polynomial associated with the graph $G$.

\begin{definition}
\label{def:computable}
A graph polynomial
{\em $\bP$ is computable} if 
\begin{enumerate}[(i)]
\item
$\bP$ is a Turing computable function,
and additionally, 
\item
the range of $\bP$, the set
$$
\{ p(\bX) \in \Z[\bX] : \mbox{ there is a graph  } G \mbox{  with  } \bP(G;\bX) = p(\bX) \}
$$
is Turing decidable.
\end{enumerate}
\end{definition}

In this paper we give a general formulation to Wilf's question.

\begin{problem}[Recognition and Characterization Problem:]
Given a graph polynomial $\bP(G;\bX)$ and a graph property $\cC$, define
$$
\cY_{\bP, \cC} = \{ p(\bX) \in \Z[\bX] : \exists G \in \cC \mbox{  with  } \bP(G;\bX) = p(\bX) \}
$$
\begin{enumerate}[(i)]
\item
The {\em recognition problem} asks for an algebraic method to decide membership
in $\cY_{\bP, \cC}$.
\item
The {\em characterization problem}
asks for
an {\em algebraic characterization of $\cY_{\bP, \cC}$}, i.e., an algebraic characterization 
of the coefficients of $p(\bX)$.
\end{enumerate}
\end{problem}
Both the recognition and the characterization problem
were stated explicitly for the chromatic polynomial $\chi(G;X)$ and $\cC$ 
the class of all finite graphs by H. Wilf, \cite{ar:Wilf1973}, 
and he deemed them to be very difficult. 

When H. Wilf asked the question about the chromatic polynomial he had an algebraic and descriptive answer
in mind. 
Something like, a polynomial $p(X)$ is a chromatic polynomial of a some graph $G$ iff
the coefficients satisfy some relations. The conjecture, that the absolute values of the coefficients
of the chromatic polynomial
form a unimodal sequence, only recently proved by  J. Huh, \cite{ar:Huh2015} has its origin in Wilf's question.
H. Wilf was not concerned about algorithmic complexity. 

From a complexity point of view, we note that deciding whether a given polynomial $p(X)$ is a chromatic polynomial of 
a some graph $G$ can be decided by brute force in exponential time as follows:
\begin{enumerate}[(i)]
\item
Use the degree $d_p$ of $p(X)$ to determine the upper bound on the size of the candidate graph $G$.
In the case of the chromatic polynomial we have $|V(G)| = d_p$.
\item
Let $I(n)$ be the number of graphs, up to isomorphism, of order $n$.
Listing all graphs, up to isomorphism, of order $n$, 
is exponential in $n$.
\item
For $i \leq I(d_p)$ compute the chromatic polynomial $\chi(G_i; X)$ and test if $p(X) = \chi(G_i; X)$.
Evaluating $\chi(G_i; X)$ for $X=a$ and  $a \in \N$ is  in $\sharp\bP$.
\end{enumerate}
The same argument works for many  other graph polynomials. 

\begin{problem}[Algorithmic version of Wilf's problem:]
Given a graph polynomial $\bP$ and a graph property $\cC$,
determine the complexity of the recognition problem for $\cY_{\bP, \cC}$.
\end{problem}

One can view  the result of
C.D. Godsil and I. Gutman, \cite{ar:GodsilGutman1981} as a solution of a very special case
of Wilf's  Characterization Problem, where $\bP = \mu(G;X)$ is the matching polynomial, and $\cC$ is
the the indexed family of
$P_n$, cycles $C_n$, complete graphs $K_n$
and bipartite complete graphs $K_{n,n}$.
The solution to the Recognition Problem is then given by verifying that the polynomial $p(X)$ in question
satisfies a recurrence relation.
We shall discuss this and generalizations thereof in Section \ref{se:recur}.

In this paper we are interested in two questions:

\begin{description}
\item[({\bf A}):]
How can we get solutions to Wilf's Characterization Problem for a general class of graph polynomials?
\item[({\bf B}):]
Given such a solution, what does it say about the underlying graphs?
\end{description}

Using classical results and a general theorem from \cite{ar:FischerMakowsky08},
this paper gives solutions to ({\bf A}) similar to the characterization  of
C.D. Godsil and I. Gutman, \cite{ar:GodsilGutman1981}, 
for a large class of graph polynomials
and indexed families of graphs $G_n$
by replacing, in many cases, the orthogonal polynomials by polynomials given by other 
linear recurrence relations with constant
coefficients.
We shall see in Section \ref{se:abstract-theorem} how to formulate a meta-theorem which captures 
many cases for special classes of graphs.

\subsection{Algebraic vs graph theoretic properties of graph polynomials}
As for ({\bf B}), the answer depends on the
particular way the graph polynomial is represented.
The situation is comparable to linear algebra, with matrices and linear maps $f: V \rightarrow W$
between vector spaces $V$ and $W$.
If we choose bases in $V$ and $W$, we can associate with $f$ a matrix $M_f$ representing $f$.
Two similar matrices represent the same linear map in terms of different choices of bases.
As every matrix is similar to a triangular matrix, triangularity is a property of the matrix $M_f$
and not of the linear map $f$.
However, $\det(M_f) = 0$ is both a property of $M_f$ and of the linear map:
$f$ is singular iff $\det(M_f)=0$ iff $\det(M)=0$ for every matrix $M$ similar to $M_f$.

Let $g_n(X)$ be a polynomial basis of $\Z[X]$.
A graph polynomial is always written in the form
$$
P(G;X) = \sum_i a_i(G) \cdot g_i(X)
$$
where the coefficients are graph parameters. 
The graph polynomial $P(G;X)$ defines an equivalence relation on the the class of finite graphs:
Two graphs $G_1, G_2$ are $P$-equivalent iff  $P(G_1;X)= P(G_2;X)$.
The various equivalence relations induced by a graph polynomial $P(G;X)$ are partially ordered by the refinement relation.
In analogy to {\em similarity of matrices}, we say that two graph polynomials have the {\em same distinctive power},
or are {\em d.p.-equivalent},
if they induce the same equivalence relation on graphs with the same number of vertices, edges and connected components.
A property of a graph polynomial is a {\em semantic (aka graph theoretic) property} if it is invariant under d.p.-equivalence.
Otherwise it is a property of the representation, i.e., the choice of the polynomial bases,
and we speak of {\em syntactic (aka algebraic)} properties of the graph polynomial.
A more detailed treatment is given in Section \ref{se:gp}.
Semantic properties of a graph polynomial $\bP$ cannot be expressed in terms of 
algebraic properties of $\cY_{P, \cC}$ alone without relating to the particular form of $\bP$.
For a thorough discussion of this, cf. \cite{ar:MakowskyRavveBlanchard2014}.
There we argued that determining the location of the roots of a graph polynomial is not a semantic property
of graph polynomials. We showed that every graph polynomial can be transformed with mild transformations
into a d.p.-equivalent graph polynomial with roots almost wherever we want them to be. 
%XXX

Here, in contrast to \cite{ar:MakowskyRavveBlanchard2014},  
we will focus on several classes  of very naturally defined graph polynomials,
the generalized chromatic polynomials
and polynomials defined as generating functions of induced or spanning subgraphs, and determinant polynomials. 
Restricting the form of the graph polynomials means restricting the coefficients of the graph polynomial in a
way that allows a natural combinatorial interpretation.
We show that for $\mathcal{P}$ either
the generalized chromatic polynomials or the
graph polynomials defined as generating functions of induced or spanning subgraphs,
the algebraic properties of the resulting graph polynomials
are semantic properties
in the sense that for every graph polynomial $\bP \in \mathcal{P}$
there is exactly one different graph polynomial $\bQ \in \mathcal{P}$ which is d.p.-equivalent to $\bP$, 
(Theorem \ref{th:char-1}).
In other words, we
give a characterization of d.p.-equivalence for graph polynomials in a particular simple form.

\subsection{Main results}
Our main contributions in this paper are more conceptual than technical. We put Wilf's recognition and
characterization problem, originally formulated for the chromatic polynomial, into the general framework of the
systematic study of graph polynomials.
To do this, we reinterpret diverse results from the literature into this general framework. This leads us
to the following results:
%---------

Let $\cC= \{G_n: n \in \N \}$ be a family of graphs and $\bP$ a graph polynomial.
\begin{enumerate}[(i)]
\item
We use a general criterion from \cite{ar:FischerMakowsky08}
to show that the sequence of polynomials $\bP(G_n;\bX)$ satisfying a 
linear recurrence relation with constant coefficients is C-finite\footnote{
The terminology C-finite is usually used for sequences of natural numbers,
and we adopt it here for polynomials, cf. \cite{bk:A=B}.
},
(Theorem \ref{th:c-finite}).
\item
Let $\bP$ be a graph polynomial.
A graph $G$ is $\bP$-unique if whenever for a graph $H$ we have $\bP(H;X) = \bP(G;X)$
then $H$ is isomorphic to $G$.
If every $G_n$ is $\bP$-unique we use this recurrence relation
to give 
characterizations of $\cY_{P, \cC}$ for many graph polynomials,
(Theorem \ref{th:abstract}).
\item
Graph polynomials are compared by their respective distinctive power
(d.p.- and s.d.p.-equivalence).
We characterize d.p.- and s.d.p.-equivalence, Proposition \ref{th:dp}.
For d.p.-equivalent graph properties 
$\cC, \cD$ 
this gives:
$\cC$ and
$\cD$ 
are d.p.-equivalent iff $\cC= \cD$ or $\cC= \cG- \cD$, 
(Proposition \ref{th:char-0})
\item
In Section \ref{se:linalg}
we study d.p.- and s.d.p.-equivalence of graph polynomials 
$\bP_{\cC}^{ind}(G;X)$ and
$\bP_{\cD}^{span}(G;X)$
obtained as generating functions for induced
and spanning subgraphs
in more detail.
They refine the d.p.-equivalence  relation of the respective graph properties $\cC$ and $\cD$.
Theorems 
\ref{th:dp-incomparable} and
\ref{th:dp-incomparable-span}
show
that there are infinitely many mutually d.p.-incomparable graph polynomials of this form.
\item
Theorem \ref{th:char-C-finite} shows that C-finiteness is a semantic property of 
graph polynomials obtained as generating functions for induced
and spanning subgraphs.
\item
Also in Section \ref{se:linalg}
we study d.p.-equivalence of generalized chromatic polynomials 
$\chi_{\cC}(G;X)$.
They also refine the d.p.- and s.d.p.-equivalence  relation of graph property $\cC$.
Theorem \ref{th:dp-incomparable-chrom} states that
there are infinitely many mutually s.d.p.-incomparable graph polynomials of this form.
\item
Finally, we consider graph polynomials which are generating functions of relations $\bP_{\Phi(A)}(G;X)$, and 
we show that not every graph polynomial  of this form can be written as a generating function of induced (spanning)
subgraphs or as a generalized chromatic polynomial $\chi_{\mathcal{C}}(G;X)$,
Theorems \ref{th:dominating} and \ref{th:notchrom}.
\end{enumerate}

%----------------------------------------------------------------
\ifsubmit\else
\begin{framed}
\noindent
File: NREV-gp
\\
Label:  se:gp
\end{framed}
\fi %submit
\section{How to define and compare graph polynomials?}
\label{se:gp}
\subsection{Typical forms of graph polynomials}
In this paper we look at five types of graph polynomials: generalized chromatic polynomials
and polynomials defined as generating functions of induced or spanning subgraphs, and determinant polynomials,
and contrast this to graph polynomials arising from generating functions of relations.

More precisely,
let $\cC$ be a graph property. 
\begin{description}
\item[Generalized chromatic:]
Let $\chi_{\cC}(G;k)$ denote the number of colorings of $G$
with at most $k$ colors such that each color class induces a graph in $\cC$.
It was shown in \cite{ar:KotekMakowskyZilber08,ar:KotekMakowskyZilber11} that 
$\chi_{\cC}(G;k)$ is a polynomial in $k$
for any graph property $\cC$.
Generalized chromatic polynomials are further studied in \cite{ar:GHKMN-2016}.
\item[Generating functions:]
Let $A \subseteq V(G)$ and $B \subseteq E(G)$. We denote by $G[A]$ the induced subgraph of $G$
with vertices in $A$, and by $G\langle B \rangle$ the spanning subgraph of $G$ with edges in $B$.
\begin{enumerate}[(i)]
\item
Let $\cC$ a graph property.
$$
P_{\cC}^{ind}(G;X)= \sum_{A \subseteq V: G[A] \in \cC} X^{|A|}
$$
\item
Let $\cD$ a graph property which is closed under adding isolated vertices, i.e.,
if $G \in \cD$ then $G \sqcup K_1 \in \cD$.
$$
P_{\cD}^{span}(G;X)= \sum_{B \subseteq E: G<B> \in \cD} X^{|B|}
$$
\end{enumerate}
%-------------
%\marginpar{Should we skip what is in the frame? It is not used later.}
%\begin{framed}
\item[Generalized Generating functions:]
Let $X_i: i \leq r$ be indeterminates and $f_i: i \leq r$ be graph parameters.
We also consider graph polynomials of the form
$$
P_{\cC, f_1, \ldots, f_r}^{ind}(G;X)= \sum_{A \subseteq V: G[A] \in \cC} \prod_{i=1}^r X_i^{f_i(G[A])}
$$
and
$$
P_{\cC, f_1, \ldots, f_r}^{span}(G;X)= \sum_{B \subseteq E: G<B> \in \cD} \prod_{i=1}^r X_i^{f_i(G<B>)}
$$
%\end{framed}
\item[Determinants:]
Let $M_G$ be a matrix associated with a graph $G$, such as the adjacency matrix, the Laplacian, etc.
Then we can form the polynomial $\det(\mathbf{1} \cdot X - M_G)$.
\end{description}
Special cases are 
the chromatic polynomial $\chi(G;X)$, 
the independence  polynomial $I(G;X)$, 
the Tutte
polynomial $T(G;X,Y)$ and
the characteristic polynomial of a graph $p_{char}(G;X)$.
Note that, in the sense of the following subsection, 
$\chi(G;X)$, $I(G;X)$ and $p_{char}(G;X)$ are mutually d.p.-incomparable, and $\chi(G;X)$
has strictly less distinctive power than $T(G;X,Y)$.

In Section \ref{se:relation} we shall see that there are graph polynomials defined in the
literature which seemingly do not fit the above frameworks. 
This is the case for the usual definition of
the {\em generating matching polynomial}: 
$$
\sum_{M \subseteq E(G): \mbox{match}(M)} X^{|M|}
$$
where $\mbox{match}(M)$ says that $V(G),M$ is a matching.
However, we shall see in Section \ref{se:relation} that there is another definition of the same polynomial
which is an generating function. In stark contrast to this, we shall prove there,
that the dominating polynomial 
$$
DOM(G;X)  = \sum_{A \subseteq V(G): \Phi_{dom}(A)} X^{|A|}
$$ 
where $\Phi_{dom}(A)$ says that $A$ is a dominating set of $G$,
cannot be written as a generating function, (Theorem \ref{th:dominating}).
This motivates the next definition, see also Section \ref{se:relation}.

\begin{description}
\item[Generating functions of a relation]
Let $\Phi$ be a property of pairs $(G,A)$ where $G$ is a graph and $A \subseteq V(G)^r$
is an $r$-ary relation on $G$.
Then the generating function of $\Phi$ is defined by
$$
\bP_{\Phi}(G;X) =
\sum_{A \subseteq V(G)^r: \Phi(G,A)} X^{|A|}
$$
\item[The most general graph polynomials]
Further generalizations of chromatic polynomials were studied
in
\cite{ar:MakowskyZilber2006,ar:KotekMakowskyZilber11,phd:Kotek}
%and in \cite{ar:Goodall-etal-1,ar:Goodall-etal-2}.
and in \cite{ar:GarijoGoodallNesetril2013,GNOdM16}.
In \cite{ar:MakowskyZilber2006,ar:KotekMakowskyZilber11}
it was shown that the most general graph polynomials can be obtained
using model theory as developed in \cite{Zilber-uct,ar:CherlinHrushovski}.
A similar approach was used in
\cite{ar:GarijoGoodallNesetril2013,GNOdM16} based on ideas
from \cite{ar:delaHarpeJaeger1995}.
However, for our presentation here, the graph polynomials we have defined so far
suffice.
\end{description}

%-----------------------------------------------------------

\subsection{Comparing graph polynomials}

We denote by $n(G), m(G), k(G)$ the number of vertices, edges and connected components of $G$.
Let $\bP$ be a graph polynomial.
A graph $G$ is {\em $\bP$-unique}  if every graph $H$ with $\bP(H;\bX) = \bP(G;\bX)$
is isomorphic to $G$.
We say that
two graphs $G,H$ are {\em similar} if the have the same number of vertices, edges and connected components.
Two graphs $G,H$ are {\em $\bP$-equivalent} if $\bP(H;\bX) = \bP(G;\bX)$.
{\em $\bP$  distinguishes between $G$ and $H$} if $G$ and $H$ are not $\bP$-equivalent.

Two graph polynomials $\bP(G;\bX)$ and $\bQ(G;\bY)$ with $r$ and $s$ indeterminates respectively
can be compared by their {\em distinctive power} on similar graphs:
{\em $\bP$ is at most as distinctive as $\bQ$}, $\bP \leq_{s.d.p} \bQ$ if any two similar graphs $G,H$ which are
$\bQ$-equivalent are also $\bP$-equivalent.
{\em $\bP$  and $\bQ$ are s.d.p.-equivalent}, $\bP \sim_{s.d.p} \bQ$ if for any two similar graphs $G,H$
$\bP$-equivalence and
$\bQ$-equivalence coincide.
%-------------addition A-------------
We can also compare graph polynomials on graphs without requiring similarity.
In this case we say that a graph polynomial $\bP$ is at most as distinctive as $\bQ$, $\bP \leq_{d.p.} \bQ$,
if for all graphs $G_1$ and $G_2$ we have that 
$$
\bQ(G_1) = \bQ(G_2) \mbox{   implies   }
\bP(G_1) = \bP(G_2)
$$
$\bP$ and $\bQ$ are d.p.-equivalent iff both
$\bP \leq_{d.p.} \bQ$ and
$\bQ \leq_{d.p.} \bP$.
D.p.-equivalence is stronger that s.d.p.-equivalence:

\begin{lemma}
\label{le:new}
For any two graph polynomials $\bP$ and $\bQ$ we have:
$\bP \leq_{d.p.} \bQ$ implies
$\bP \leq_{s.d.p.} \bQ$. 
\end{lemma}
In this paper we concentrate on d.p.-equivalence and speak of s.d.p.-equivalence only when it is needed.
In a sequel to this paper we will investigate in detail what can be said of both notions.

%----------------end addition A-----------------------------------
Part (ii) of the following Proposition
was shown in \cite{ar:MakowskyRavveBlanchard2014},
and (iii) follows from Definition \ref{def:computable}(i) and (ii).

\begin{proposition}
\label{th:dp}
\begin{enumerate}[(i)]
\item
%-----new
$\bP$ is at most as distinctive as $\bQ$, $\bP \leq_{d.p} \bQ$, iff there is a function
$F: \Z[\bY] \rightarrow \Z[\bX]$ such that for every graph $G$ we have
$$
\bP(G;\bX) = F(\bQ(G;\bY))
$$
%--------
\item
$\bP$ is at most as distinctive as $\bQ$ on similar graphs, $\bP \leq_{s.d.p} \bQ$, iff there is a function
$F: \Z[\bY] \times \Z^3 \rightarrow \Z[\bX]$ such that for every graph $G$ we have
$$
\bP(G;\bX) = F(\bQ(G;\bY), n(G), m(G), k(G))
$$
\item
Furthermore, both for d.p. and s.d.p., if both $\bP$ and $\bQ$ are computable, then $F$ is computable, too.
\end{enumerate}
\end{proposition}
\begin{proof}[Proof of (iii):]
The function $F$ from (i) or (ii) is not unique.
However, because the range of polynomials given by $\bQ$ is assumed to be decidable,
we can choose $F$ such that $F(p) = 0$ for all $p \in \Z[\bX]$ such that there is no graph $G$ with $\bQ(G;\bX) =p$.
$F$ chosen in this way now is computable.
\end{proof}
%-----------------------addition B--------------
\begin{remark}
In the literature \cite{ar:MerinoNoble2009,ar:Sokal2005}
on the Tutte polynomial s.d.p.-equivalence is  implicitly used to compare the various
forms
of the Tutte polynomial and the Potts model.
The various forms of the Tutte polynomial are not d.p.-equivalent.
The same is true for the various forms of the matching polynomial as discussed in, say, \cite{bk:LovaszPlummer86}.
\end{remark}
%-----------------------
%The general definitions are given in Section \ref{se:gp}.
%----------------------------------

Graph polynomials are supposed to give information about graphs.
The algebraic characterization of $\cY_{\bP, \cC}$ uses the coefficients of the  polynomial $\bP(G;\bX)$
to characterize $\cC$. However, such a characterization depends on the presentation of $\bP(G;\bX)$
with respect to the basic polynomials chosen to write $\bP(G;\bX)$. In the univariate case, the basic polynomials
are usually, but not always, $X^n$. Sometimes one uses ${X \choose n}$ instead, or the falling factorial
$X_{(n)} = X \cdot (X-1) \cdot \ldots \cdot (X-n+1)$.
On the other side, the notion of d.p.-equivalence captures properties of graphs independently 
of the presentations of the particular graph polynomials.
A statement involving a graph polynomial $\bP$ is a proper statement about graphs, 
if it is invariant under d.p.-equivalence.
Otherwise, it is merely a statement about graphs via the particular presentation of the graph polynomial.

\begin{problem}[Invariance under d.p.-equivalence]
Are there algebraic characterizations of $\cY_{\bP, \cC}$ which are invariant under d.p.- or s.d.p.-equivalence?
\end{problem}
In \cite{ar:MakowskyRavveBlanchard2014},
this question was studied concerning the location of the roots of $\bP(G;\bX)$. 
The answer was negative
even if the class of graph polynomials considered is closed under substitutions, and prefactors.
However,
the various versions of the Tutte polynomials and matching polynomials  can be obtained from
each other in this way. 

In the light of the discussion in \cite{ar:MakowskyRavveBlanchard2014}, we now look at the
following problem:

\begin{problem}[Distinctive power]
Given two graph polynomials $\bP, \bQ$ in a specific form such as 
generalized chromatic polynomials,
polynomials defined as generating functions of induced or spanning subgraphs, or determinant polynomials,
characterize when they are d.p-equivalent. 
\end{problem}

In Section \ref{se:linalg}, we will discuss this problem for the case of 
polynomials defined as generating functions of induced subgraphs
and
generalized chromatic polynomials, Theorem \ref{th:char-1} and  Proposition \ref{prop:char-2}.
These theorems
do not hold for generating functions of relations,
see Section \ref{se:relation}.

%-----------------------------------------
\ifsubmit\else
\begin{framed}
\noindent
File: NREV-recog
\\
Label:  se:recog
\end{framed}
\fi %submit
%-------------------------------------------------------------
\section{The recognition and characterization problems}
%-------------------------------------------------------------
\label{se:recog}

Let $\bP(G;\bX)$ be a computable 
%$\SOL$-definable 
graph polynomial, 
and let $\cC$ be a graph property.
Recall from 
Section \ref{se:intro},
that the {\em Recognition Problem for $\bP(G;\bX)$ and $\cC$} is the question, whether,
given a polynomial $s(\bX) \in \Z[\bX]$,
there is a graph $G_s \in \cC$ such that $\bP(G_s;\bX)=s(\bX)$?
The {\em Characterization Problem for $\bP(G;\bX)$ and $\cC$} asks for a description of the set
of polynomials
$$
\cY_{\bP, \cC} = \{ p(\bX) \in \Z[\bX] : \exists G \in \cC \mbox{   with   } \bP(G;\bX) = p(\bX) \}.
$$
If $\cC$ is the class of all finite graphs, we write $\cY_{\bP}$. 
We also noted in 
Section \ref{se:intro}, that there is a brute force solution for the Recognition Problem as follows:
\begin{observation}
Assume that $\bP(G;\bX)$ is a computable graph polynomial for which 
we can give a bound $\beta(\bP(G;\bX))$ for the size of $G$.
Then, checking whether a polynomial $p(\bX)$ is in $\cY_{\bP}$ can be done by computing $\bP(G;\bX)$
for all graphs smaller than  $\beta(\bP(G;\bX))$.
\end{observation}
What we are looking for should be better than that.

The problem may be easier for certain graph properties $\cC$ in the relative version.
There are many such characterization in the literature, we just give one here for the sake of illustration.
\begin{example}[Taken from \cite{bk:DongKohTeo2005}]
Let $\cC$ be the class of finite connected graphs. Then $\cY_{\chi, \cC}^r$ consists
of all instances of the chromatic polynomial which have $0$ as a root with multiplicity one.
\end{example}

It is easy to define graph polynomials $P(G;X)$ with a {\em trivial recognition}, i.e., where for every
polynomial 
$$s(X) = \sum_{i=0}^m a_i X^i \in \N[X]$$ 
there is a graph $G_s$ with $P(G_s;X)= s(X)$.
\begin{proposition}
Let 
$MaxCl(G;X) = \sum_i mcl_i(G) X_i$
be the  graph polynomial 
where $mcl_i(G)$ denotes the number of maximal cliques of size $i$.
$MaxCl(G;X)$ has a trivial recognition.
\end{proposition}
\begin{proof}
Let $s(X) = \sum_{i=0}^m a_i X^i \in \N[X]$ and let $G_s$ be the graph which is the disjoint union
of $a_i$-many cliques of size $i$. Then $MaxCl(G_s;X) = s(X)$.
\end{proof}

\begin{Problem}
Find more naturally defined graph polynomials with trivial recognition.
\end{Problem}

To show that {\em not all polynomials} are chromatic polynomials, 
one can use various properties of the coefficients.
One sufficient condition is that the coefficients are alternating in sign for connected graphs.
For a characterization, more properties of the coefficients are needed.
One such property is 
the fact that its coefficients (or their absolute values)
are unimodal or logconcave, cf. \cite{ar:Huh2015}, 
which was suggested also for the
independence polynomial, and other graph polynomials.
%cf. Section \ref{se:unimodal}.
However, showing unimodality of the coefficients is notoriously hard, 
\cite{ar:Stanley1989,ar:Brenti1992,handbook:branden2015,ar:Huh2015}.

%\input{FPSAC-2016/fpsac-recognition}
%----------------------------------------
\ifsubmit\else
\begin{framed}
\noindent
File: NREV-recur
\\
Label:  se:recur
\end{framed}
\fi %submit
%-------------------------------------------------------------
\section{Characterizations using recurrence relations}
\label{se:recur}
%-------------------------------------------------------------
Let $P_n$, $C_n$ and $K_n$ 
denote, respectively, the path, the cycle and the complete graph on $n$ vertices,
and $K_{n,m}$ denote the complete bipartite graph on $n+m$ vertices.
We let $\mathrm{Path}$ be the family of paths $P_n$, and $\mathrm{Cycle}, \mathrm{Clique}$ and 
$\mathrm{CBipartite}$ the families of $C_n$, $K_n$ and $K_{n,n}$ respectively.
For a class of graphs $\cC$ closed under isomorphisms, we denote the class of graphs
consisting of disjoint unions of graphs in $\cC$ by $\DU(\cC)$.

Let $\bP(G;\bX)$ be a graph polynomial and $G_n$ be a sequence of graphs.
The {\em sequence of polynomials $\bP_n(\bX)=\bP(G_n;\bX)$ is C-finite}
if there is $q \in \N$ and there are polynomials $f_i(\bX) \in \Z[\bX], i \in [q]$ such that
$$
\bP_{n+q}(X) = \sum_{i =0}^{q-1}  f_i(\bX) \bP_i(\bX)
$$

\subsection{The characteristic polynomial $p_{char}(G;X)$}
Let $G$ be an undirected graph and $A_G$ is symmetric adjacency matrix.
The characteristic polynomial is defined as 
$$
p_{char}(G;X) = \det(\mathbf{1} \cdot X -A_G)
$$
We note that  $p_{char}(G;X)$ is {\em multiplicative}, i.e.,
if $H$ is the disjoint union of $G_1$ and $G_2$ then 
$p_{char}(H;X) = p_{char}(G_1;X) \cdot p_{char}(G_2;X)$.  

\begin{proposition}[Taken from {\cite[Chapter 14.4.2]{bk:BrouwerHaemers2012}}]
\label{pr:char-unique}
\ \\
The graphs $P_n$, $C_n$, $K_n$ and $K_{n,n}$ are $p_{char}$-unique.
\end{proposition}

\begin{proposition}[A.J. Schwenk \cite{ar:Schwenk1974}]
\label{pr:char-c-finite}
\ \\
The sequences of polynomials $p_{char}(P_n)$, $p_{char}(C_n)$, 
$p_{char}(K_n)$ and $p_{char}(K_{n,n})$ are all C-finite.
\end{proposition}

Proposition \ref{pr:char-unique} 
%and \ref{pr:char-c-finite}
gives us:
\begin{theorem}
\ 
\begin{enumerate}[(i)]
\item
$G$ is isomorphic to $P_n$ iff $p_{char}(G;X) = p_{char}(P_n;X)$. 
\item
$G$ is isomorphic to $C_n$ iff $p_{char}(G;X) = p_{char}(C_n;X)$. 
\item
$G$ is isomorphic to $K_n$ iff $p_{char}(G;X) = p_{char}(K_n;X)$. 
\item
$G$ is isomorphic to $K_{n,n}$ iff $p_{char}(G;X) = p_{char}(K_{n,n};X)$. 
\end{enumerate}
\end{theorem}
Proposition \ref{pr:char-c-finite}
gives us similar
characterizations of
$\cY_{p_{char}, \mathrm{Path}}, \cY_{p_{char}, \mathrm{Cycle}}, \cY_{p_{char}, \mathrm{Clique}}$ and
$\cY_{p_{char}, \mathrm{CBipartite}}$  using C-finiteness.

%---------------------------------------------
\ifskip\else
Using the multiplicativity of $p_{char}$ we get a characterization of
$\cY_{p_{char}, \DU(\mathrm{Path})}$: 
\begin{theorem}
$G$ is isomorphic to the disjoint union of $P_{n_1}, \ldots P_{n_k}$ iff 
$p_{char}(G;X) = \prod_{i=1}^k p_{char}(P_{n_i};X)$. 
\end{theorem}
The same works for 
$\DU(\mathrm{Cycle})$,
$\DU(\mathrm{Clique})$ and
$\DU(\mathrm{CBipartite})$.
\fi %skip
%---------------------------------------------
\subsection{The matching polynomials}
Let $m_k(G)$ denote the number of $k$-matchings of a graph $G$ on $n$ vertices.
Let $\mu(G;X)$ be the defect matching polynomial (aka acyclic polynomial)
$$
\mu(G;X) = \sum_{k=0}^{\lfloor n/2 \rfloor} m_k(G) (-1)^k X^{n-2k}
$$
\begin{theorem}[C.D. Godsil and I. Gutman \cite{ar:GodsilGutman1981}]
On forests $F$ we have $\mu(F;X) = p_{char}(F;X)$.
\end{theorem}
We look for characterizations of
$
\cY_{\mu, \mathrm{Path}},
\cY_{\mu, \mathrm{Cycle}},
\cY_{\mu, \mathrm{Clique}}$ and
$\cY_{\mu, \mathrm{CBipartite}}$.

We need the recursive definitions of the orthogonal polynomials of
Chebyshev, Hermite and Laguerre, cf. \cite{bk:Chihara2011}:
The Chebyshev polynomials $T_n(X)$ and $U_n(X)$ are defined recursively as follows:
$$
T_0(X)=1, T_1(X)=X \mbox{   and   } T_{n+1}(X) = 2X\cdot T_n(X) - T_{n-1}(X)
$$
and
$$
U_0(X)=1, U_1(X)=2X \mbox{   and   } U_{n+1}(X) = 2X\cdot U_n(X) - U_{n-1}(X)
$$
These two recurrence relations are linear in $U_n(X)$ and
their coefficients 
are elements of $\N[X]$ and {\em do not depend on $n$}.

The Hermite polynomials 
%$H_n(X)$ and
$He_n(X)$
are defined recursively as follows:
$$
He_0(X)=1, He_1(X)=X \mbox{   and   } He_{n+1}(X) = X\cdot He_n(X) - n \cdot He_{n-1}(X)
$$
%and
%$$
%H_0(X)=1, H_1(X)=2X \mbox{   and   } H_{n+1}(X) = 2X\cdot H_n(X) - 2n\cdot H_{n-1}(X)
%$$
The Laguerre polynomials $L_n(X)$ are defined recursively as follows:
$$
L_0(X)=1, L_1(X)=1-X \mbox{   and   } L_{n+1}(X) = \frac{2n+1-x}{n+1} \cdot L_n(X) - \frac{n}{n+1} L_{n-1}(X)
$$
These recurrence relation are linear in $He_n(X)$ and $L_n(X)$ and
their coefficients 
are elements of $\N[X]$, respectively $\Q[X]$ and {\em do depend on $n$}.

\begin{theorem}[C.D. Godsil and I. Gutman \cite{ar:GodsilGutman1981}]
\label{th:chebyshev}
\ 
\begin{enumerate}[(i)]
\item
$\mu(C_n; 2X) = 2\cdot T_n(X)$
\item
$\mu(P_n; 2X) = U_n(X)$
\item
$\mu(K_n; X) = He_n(X)$
\item
$\mu(K_{n,n}; X) = (-1)^n \cdot L_n(X^2)$
\end{enumerate}
\end{theorem}
For related theorems, cf. also \cite{ar:Godsil1981a,ar:Gessel1989,ar:NoyRibo04,ar:DiaconisGamburd2004}.

\begin{theorem}
\label{th:unique}
\ 
\begin{enumerate}[(i)]
\item
(\cite{ar:BeezerFarrell1995})
The $C_n$'s are $\mu$-unique.
\item
(\cite{ar:Noy03})
$K_n$ and $K_{n,n}$ are $\mu$-unique.
\item
(\cite[Proposition 14.4.6]{bk:BrouwerHaemers2012})
$P_n$ is $p_{char}$-unique, and, as it is a tree, also $\mu$-unique.
\end{enumerate}
\end{theorem}

Putting all this together we get:
\begin{theorem}
\ 
\begin{enumerate}[(i)]
\item
A graph $G$ is isomorphic to a cycle $C_n$ iff
$\mu(G;X) = 2\cdot T_n(X)$. In other words, $\cY_{\mu, \mathrm{Cycle}}$
can be characterized using a linear recurrence relation with constant coefficients in $\Z[X]$. 
\item
A graph $G$ is isomorphic to a path $P_n$ iff
$\mu(G;2X) = U_n(X)$. In other words, $\cY_{\mu, \mathrm{Path}}$
can be characterized using a linear recurrence relation with constant coefficients in $\Z[X]$.
\item
A graph $G$ is isomorphic to a complete graph $K_n$ iff
$\mu(G;X) = He_n(X)$. In other words, 
$\cY_{\mu, \mathrm{Clique}}$
can be characterized using a recurrence relation where the coefficients depend on $n$.
\item
A graph $G$ is isomorphic to a complete bipartite graph $K_{n,n}$ iff
$\mu(G;X) = (-1)^n \cdot L_n(X^2)$. In other words, $\cY_{\mu, \mathrm{CBipartite}}$
can be characterized using a recurrence relation where the coefficients depend on $n$.
\end{enumerate}
\end{theorem}

%------------------------------------------------------------------
\ifskip\else
A special case of the Main Theorem  of \cite[th:main]{ar:FischerMakowsky08}, together with 
Theorem \ref{th:unique}
gives us
\begin{theorem}
\label{th:recurrence}
Both, $\cY_{\mu, \mathrm{Clique}}$ and $\cY_{\mu, \mathrm{CBipartite}}$
can be characterized using a linear recurrence relation with constant coefficients in $\Z[X]$.
In other words, there are $q,r \in \N$ and polynomials $f_i(X), g_j(X) \in \Z[X]$ for $i \leq q, j \leq r$
such that
$$
\mu(K_{n+q+1};X) = \sum_{i \leq q} f_i(X) \cdot \mu(K_{n+i};X)
$$
and
$$
\mu(K_{n+q+1, n+q+1};X) = \sum_{j \leq q} g_j(X) \cdot \mu(K_{n+j,n+j};X)
$$
\end{theorem}

\begin{remarks}
\ 
\begin{enumerate}[(i)]
\item
More generally,  given a graph polynomial $P(G;\bX)$ and a sequence $G_n$ of graphs
the authors of \cite{ar:FischerMakowsky08} formulate  sufficient conditions on the form of $P(G;\bX)$ and $\{G_n\}$ such that
the polynomials  $P(G_n;\bX)$
satisfy
a linear recurrence relation with constant coefficients in $\Z[X]$.
The conditions are based on definability criteria in logic. 
If additionally all the graphs $G_n$ are $P$-unique we get a characterization as follows:
Let $p_n(X) = P(G_n;X)$. Then $P(H;X) = p_n(X)$ iff $H$ is isomorphic to $G_n$.
\item
The recurrence relation characterizing  $\cY_{\mu, \mathrm{Clique}}$ and $\cY_{\mu, \mathrm{CBipartite}}$
with Hermite and Laguerre polynomials is short (we need only the two immediate previous values) and explicit, 
but the coefficients depend on $n$.
\item
Theorem \ref{th:recurrence} is only an existence theorem. We do not know explicit linear recurrence
relations for $\mu(K_n;X)$ or $\mu(K_{n,n};X)$ where the coefficients do not depend on $n$,
nor do we know how many previous terms are needed, only that this number also does not depend on $n$.
\end{enumerate}
\end{remarks}
\fi %skip
%------------------------------------------------------------------

%\input{FPSAC-2016/fpsac-recurrence}
%-----------------------------------------
\ifsubmit\else
\begin{framed}
\noindent
File: NREV-fima
\\
Label:  se:abstract-theorem
\end{framed}
\fi %submit
\subsection{An abstract theorem}
\label{se:abstract-theorem}
Our discussion of the characteristic polynomial can be formulated abstractly.
We state the following observation as a theorem.

\begin{theorem}
\label{th:abstract}
Let $\bP$ be a graph polynomial and $\cC=\{G_n: n \in \N \}$ be given as a sequence of graphs.
Assume the following:
\begin{enumerate}[(i)]
\item
The sequence of polynomials $\bP(G_n;\bX)$ satisfies some recurrence relation.
\item
Each $G_n$ is $\bP$-unique.
\end{enumerate}
Then
$\cY_{P, \cC}$  is characterized algebraically by the property:
%has an algebraic characterization of the form:
$H$ is isomorphic to $G_n$ iff
$\bP(H;\bX) = \bP(G_n;\bX)$. This can be checked using the recurrence relations.
\end{theorem}
One can also formulate an analogue of this theorem for families of graphs $G_{n_1, \ldots, n_k}$
depending on $k$ indices.

We look also at the following indexed families of graphs:
\begin{description}
\item[$W_n$:] The wheels $C_n \bowtie K_1$.
\item[$L_n$:] The ladders $L_n = C_n \times K_2$.
\item[$M_n$:] The M\"obius ladders $M_n$ are obtained from $C_{2n}$ by connecting any pair of opposite vertices.
\item[$C_n^2$:] The square of the cycle $C_n$ obtained by connecting any two vertices of distance two.
\item[$Grid_{n,m}$:] The square grids of size $(n \times m)$.
\end{description}
For a graph polynomial $\bP$,
an indexed family $G_n$ of graphs is $\bP$-recursive if the sequence of polynomials $\bP(G_n;\bX)$
is C-finite.
Using the main theorem from \cite{ar:FischerMakowsky08}
one can prove the following for indexed families of graphs $G_n$ of bounded tree-width.
\begin{theorem}
\label{th:c-finite}
Let $m_0$ be fixed.
The families $P_n, C_n, Grid_{n,m_0}, W_n, L_n, M_n, C_n^2$ 
are all C-finite for the graph polynomials  $p_{char}, \mu, \chi$, and $T$.
\end{theorem}
\begin{remark}
\begin{enumerate}
\item
Actually, the sequences from Theorem \ref{th:c-finite} are C-finite for every graph polynomial definable 
in Monadic Second Order Logic ($\MSOL$), such as the independence polynomial, \cite{pr:LevitMandrescu05},
and the edge elimination polynomial $\xi(G;X,Y,Z)$, \cite{ar:TittmannAverbouchMakowsky10}.
However, we do not want in this paper  to get involved with definability theory
or the formalisms of (Monadic) Second Order Logic, i.e., we want to keep it
logic-free. 
\item
The way
Theorem \ref{th:c-finite} is stated, it is non-constructive, because it does not say anything
about the form of the recurrence relation. It only asserts C-finiteness, without giving
the coefficients or the depth of the recursion.
\item
The results of \cite{ar:FischerMakowsky08} cannot be applied to
$K_n, K_{n,n}$  because the sequences of graphs $K_n, K_{n,n}$ have unbounded tree-width.
In fact the resulting families of chromatic and Tutte polynomials are not C-finite,
\cite{ar:BiggsDamerellSand72}.
In general, linear recurrence relations for a sequence of polynomials
where the coefficients depend on $n$ are not C-finite, because the coefficients may grow too fast.
\end{enumerate}
\end{remark}
The following is folklore for the chromatic polynomial and due to I. Gessel \cite{ar:Gessel1995}
for the Tutte polynomial.
\begin{theorem}
The families $K_n, K_{n,n}$ satisfy recurrence relations with the coefficients depending on $n$
for the chromatic and Tutte polynomials:
$$ \chi(K_n;X) = (X-n+1) \cdot \chi(K_{n-1};X) $$
$$ T(K_n;X,Y) = \sum_{k=1}^n { n-1 \choose k-1} (X+Y+Y^2+ \ldots + Y^{k-1}) \cdot T(K_n;1,Y)\cdot T(K_{n-k}:X,Y)$$
and similar for $K_{n,n}$.
\end{theorem}

\subsection{The chromatic and the Tutte polynomials}
To give further applications of 
Theorem \ref{th:abstract} we collect some results from \cite{bk:DongKohTeo2005,ar:MierNoy04}
on $\chi$-unique and $T$-unique graphs.

\begin{theorem}
\label{th:T-unique}
\begin{enumerate}
\item
Let $t_m$ be a tree on $m$ edges. Then $T(t_m;X) = X^m$. Hence the paths $P_n$ are neither
$\chi$-unique nor $T$-unique.
\item
$C_n, K_n, K_{n,m}$ are all $\chi$-unique, hence $T$-unique.
\item
$W_n, L_n, M_n$ and $C_n^2$ are $T$-unique but not $\chi$-unique. 
\end{enumerate}
\end{theorem}
Now, Theorem \ref{th:T-unique} allows us to give more algebraic characterizations using recurrence relations
for these sequences via the chromatic and the Tutte polynomial.

%\input{FPSAC-2016/fpsac-fima}
%-----------------------------------------
\ifsubmit\else
\begin{framed}
\noindent
File: NREV-dp
\\
Label:  se:abstract-theorem
\end{framed}
\fi %submit
%Revised February 7, 2017
\newif
\ifejc
\ejctrue
%\ejcfalse
%-------------------------------------------------------------
\section{Distinctive power}
\label{se:linalg}
\subsection{s.d.p.-equivalence and d.p-equivalence of graph properties}
A class of graphs $\cS$ which consists of all graphs having the same number of vertices, edges and connected components
is called a {\em similarity class}.

Let $\cC$ be a graph property. Two graphs $G,H$ are $\cC$-equivalent if either both are in $\cC$
or both are not in $\cC$. 
We denote by $\bar{\cC}$ the graph property $\cG - \cC$.

Therefore we have:
\begin{proposition} 
\label{pr:dp-properties}
\label{th:char-0}
\begin{enumerate}[(i)]
\item
Two
graph properties
$\cC_1$ and $\cC_2$ are d.p.-equivalent iff 
either  $\cC_1 = \cC_2$ or $\cC_1  = \bar{\cC_2}$.
\item
Two
graph properties
$\cC_1$ and $\cC_2$ are s.d.p.-equivalent iff for every similarity class $\mathcal{S}$
either  $\cC_1 \cap \cS = \cC_2 \cap \cS$ or $\cC_1 \cap \cS  = \bar{\cC_2} \cap \cS$.
\end{enumerate}
\end{proposition} 
\begin{proof}
(i):
It is straightforward that
if $\cC_2 = \cC_1 $ or $\cC_2  = \bar{\cC_1}$ then
$\cC_1$ and $\cC_2$ are d.p.-equivalent. 
\\
For the other direction, we prove first that 
$\cC_1 \subseteq \cC_2 $ or $\cC_1 \subseteq \bar{\cC_2}$.
\\
By a symmetrical argument, we then prove also
$\cC_2 \subseteq \cC_1$ or $\cC_2 \subseteq \bar{\cC_1}$,
$\bar{\cC_1} \subseteq \cC_2$ or $\bar{\cC_1} \subseteq \bar{\cC_2}$ and
$\bar{\cC_2} \subseteq \cC_1$ or $\bar{\cC_2} \subseteq \bar{\cC_1}$.
Now the result follows.
\\
(ii): Fix $\cS$. The proof is the same but relativized to $\cS$.
\end{proof}
\begin{remark}
If $\cC_1$ and $\cC_2$ are s.d.p.-equivalent it is possible that for a similarity class $\cS$ we have
$\cC_1 \cap \cS = \cC_2 \cap \cS$ but for another similarity class $\cS'$ we have
$\cC_1 \cap \cS' = \bar{\cC_2} \cap \cS'$.
\end{remark}

\begin{proposition}
\label{prop:dp-incomparable}
\begin{enumerate}[(i)]
\item
Let $\cC_1$ and $\cC_2$ be two graph properties.
Assume  that
both $\cC_1$ and $\cC_2 $ are not empty and do not contain all finite graphs,
and that $\cC_1  \neq \cC_2 $ and $\cC_1  \neq \bar{\cC_2}$.
Then $\cC_1$ and $\cC_2$ are  d.p.-incomparable, i.e., $\cC_1 \not \leq_{d.p.} \cC_2$ and $\cC_2 \not \leq_{d.p.} \cC_1$.
\item
Let $\cC_1$ and $\cC_2$ be two graph properties.
Assume there is a similarity class $\cS$ such that
both $\cC_1 \cap \cS$ and $\cC_2 \cap \cS$ are not empty and do not contain all finite graphs in $\cS$,
 and that $\cC_1 \cap \cS \neq \cC_2 \cap \cS$ and $\cC_1  \cap \cS\neq \bar{\cC_2} \cap \cS$.
Then $\cC_1$ and $\cC_2$ are  s.d.p.-incomparable, i.e., $\cC_1 \not \leq_{s.d.p.} \cC_2$ and $\cC_2 \not \leq_{s.d.p.} \cC_1$.
\end{enumerate}
\end{proposition}
\begin{proof}
We prove only (i) and leave the proof of (ii) to the reader.
Assume 
$G_1 \in (\cC_1 - \cC_2)$, 
$G_2 \in (\cC_2 - \cC_1)$ and
$G_3 \in \cC_1 \cap \cC_2 $, the other cases being similar.
Then $G_2, G_3 \in \cC_2 $. 
If $\cC_1 \leq_{d.p.} \cC_2$, we would have that 
both $G_2, G_3 \in \cC_1 $, or
both $G_2, G_3 \not \in \cC_1 $, a contradiction.
\end{proof}

In the next two subsections we look at graph polynomials, which are either generating functions, or count colorings
which, in both cases, solely depend on a graph property $\cC$.

\subsection{Graph polynomials as generating functions}

Let $\cC$ be a graph property, and $\cD$ be a graph property closed under adding and removal isolated vertices.
Recall from Section \ref{se:gp} the definitions
$$
\bP_{\cC}^{ind}(G;X)= \sum_{A \subseteq V: G[A] \in \cC} X^{|A|}
\mbox{\ \ \ \    and   \ \ \ \ }
\bP_{\cD}^{span}(G;X)= \sum_{B \subseteq E: G\langle B \rangle \in \cD} X^{|B|}.
$$
Let
$|V(G)|= n(G)$ and
$|E(G)|= m(G)$. 

\begin{proposition}
\begin{enumerate}[(i)]
\item
$\cC \leq_{d.p.} \bP_{\cC}^{ind}(G;X)$ and
\item
$\cD \leq_{d.p.} \bP_{\cD}^{span}(G;X)$. 
\end{enumerate}
\end{proposition}
\begin{proof}
(i) follows from the fact that
$G \in \cC$ iff the coefficient of $X^{n(G)}$ in $\bP_{\cC}^{ind}(G;X)$ does not vanish.
\\
Similarly, (ii) follows from the fact that
$G \in \cC$ iff the coefficient of $X^{m(G)}$ in $\bP_{\cC}^{span}(G;X)$ does not vanish.
\end{proof}

From Lemma \ref{le:new} we get immediately:
\begin{corollary}
\begin{enumerate}[(i)]
\item
$\cC \leq_{s.d.p.} \bP_{\cC}^{ind}(G;X)$ and
\item
$\cD \leq_{s.d.p.} \bP_{\cD}^{span}(G;X)$. 
\end{enumerate}
\end{corollary}

\begin{proposition}
\label{prop:complement}
With
$|V(G)|= n(G)$ and
$|E(G)|= m(G)$ we have:
\begin{enumerate}[(i)]
\item
$\bP_{\cC}^{ind}(G;X) + \bP_{\bar{\cC}}^{ind}(G;X) = %\sum_{i=0}^{n(G)} 2^i$,
(1+X)^{n(G)}$
\item
$\bP_{\cD}^{span}(G;X) + \bP_{\bar{\cD}}^{span}(G;X) = %\sum_{i=0}^{m(G)} 2^i$
(1+X)^{m(G)}$
\end{enumerate}
\end{proposition}
\begin{proof}
(i):
Put
$$
c_i(G) = | \{ A \subseteq V(G): |A|=i, G[A] \in \mathcal{C} \}|
$$
and
$$
\bar{c}_i(G) = | \{ A \subseteq V(G): |A|=i, G[A] \not\in \mathcal{C} \}|.
$$
Clearly, 
$$
c_i(G) + \bar{c}_i(G) = { n(G) \choose i},
$$
hence
$$
\sum_{i=0}^{n(G)} \left(c_i(G) + \bar{c}_i(G)\right) X^i = (1+X)^{n(G)}.
$$

(ii) is similar, but we need that for a set of edges $A \subseteq E(G)$ the spanning subgraph
$G\langle A\rangle =(V(G),A) \in \cD$ iff  $V(A),A) \in \cD$,
where $V(A) = \{ v \in V(G) : \mbox{ there is   } u \in V(G) \mbox{ with  } (u,v) \in A\}$.
\end{proof}

\ifrevised
\begin{proposition}
\label{th:char-1}
Let $\cC_1$ and $\cC_2$, $\cD_1$ and $\cD_2$ be graph properties 
such that
$\cC_1$ and $\cC_2$  and $\cD_1$ and $\cD_2$ are pairwise d.p.-equivalent,
\begin{enumerate}[(i)]
\item
\label{th:char-1.1}
$\bP_{\cC_1}^{ind}(G;X)$ and $\bP_{\cC_2}^{ind}(G;X)$
are s.d.p.-equivalent;
\item
\label{th:char-1.2}
If, additionally, $\cD_1$ and $\cD_2$ are closed under the addition and removal of isolated vertices, then
$\bP_{\cD_1}^{span}(G;X)$ and $\bP_{\cD_2}^{span}(G;X)$
are s.d.p.-equivalent;
\end{enumerate}
\end{proposition}
\begin{proof}
We prove  only (i), (ii) is proved analogously.
\\
(i): We use Proposition 
\ref{pr:dp-properties}.
If $\cC_1 =  \cC_2$, clearly, 
$\bP_{\cC_1}^{ind}(G;X) = \bP_{\cC_2}^{ind}(G;X)$, 
hence they are d.p.-equivalent.
If $\cC_1 =  \bar{\cC_2}$, we use Proposition \ref{prop:complement} together with 
Proposition \ref{th:dp}. But Proposition \ref{prop:complement} depends on the $n(G)$, 
hence we get only that 
$\bP_{\cC_1}^{ind}(G;X) and \bP_{\cC_2}^{ind}(G;X)$ 
are s.d.p.-equivalent.
\end{proof}

Let 
$G_n$ be an indexed sequence of graphs  such that
the sequence of polynomials $X^{|V(G_n)|}$ is C-finite.
This assumption is true for all the examples from Section \ref{se:abstract-theorem},
and in particular for Theorem \ref{th:chebyshev}, provided the function $|V(G_n)|$ is linear in $n$.
We shall now show that C-finiteness of the sequences of polynomials $\bP_{\cC}^{ind}(G_n;X)$
of Theorem \ref{th:chebyshev}
is a semantic property graph polynomials as generating functions.
However, the particular form of the recurrence relation is not.

\begin{theorem}
\label{th:char-C-finite}
Let $\cC_1$ and $\cC_2$, $\cD_1$ and $\cD_2$ be graph properties 
such that
$\cC_1$ and $\cC_2$  and $\cD_1$ and $\cD_2$ are pairwise d.p.-equivalent,
and let $G_n$ be an indexed sequence of graphs. 
Furthermore, assume that the sequence of polynomials $X^{|V(G_n)|}$ is C-finite.
Then
\begin{enumerate}[(i)]
\item
$\bP_{\cC_1}^{ind}(G_n;\bX)$ is C-finite iff $\bP_{\cC_2}^{ind}(G_n;\bX)$ is C-finite. 
\item
$\bP_{\cD_1}^{span}(G_n;\bX)$ is C-finite iff $\bP_{\cD_2}^{span}(G_n;\bX)$ is C-finite; 
\end{enumerate}
\end{theorem}
\begin{proof}
This follows in both cases from the fact that the sum and difference of two C-finite sequences
is again C-finite together with Proposition \ref{prop:complement}.
\end{proof}

%--------------------
\else
\begin{theorem}
\label{th:char-1}
Let $\cC_1$ and $\cC_2$, $\cD_1$ and $\cD_2$ be graph properties 
such that
$\cC_1$ and $\cC_2$  and $\cD_1$ and $\cD_2$ are pairwise d.p.-equivalent,
 and let $G_n$ be an indexed sequence of graphs. Then
\begin{enumerate}[(i)]
\item
\label{th:char-1.1}
$\bP_{\cC_1}^{ind}(G;X)$ and $\bP_{\cC_2}^{ind}(G;X)$
are d.p.-equivalent;
%are d.p.-equivalent iff $\cC$ and $\cD$ are d.p.-equivlent.
\item
$\bP_{\cC_1}^{ind}(G_n;\bX)$ is C-finite iff $\bP_{\cC_2}^{ind}(G_n;\bX)$ is C-finite. 
\item
\label{th:char-1.2}
If, additionally, $\cD_1$ and $\cD_2$ are closed under the addition and removal of isolated vertices, then
$\bP_{\cD_1}^{span}(G;X)$ and $\bP_{\cD_2}^{span}(G;X)$
are d.p.-equivalent;
%are d.p.-equivalent iff $\cC$ and $\cD$ are d.p.-equivlent.
\item
$\bP_{\cD_1}^{span}(G_n;\bX)$ is C-finite iff $\bP_{\cD_2}^{span}(G_n;\bX)$ is C-finite; 
\end{enumerate}
\end{theorem}
\begin{proof}
We prove (i) and (ii), (iii) and (iv) are proved analogously.
\\
(i): We use Proposition 
\ref{pr:dp-properties}.
If $\cC_1 =  \cC_2$, clearly, 
$\bP_{\cC_1}^{ind}(G;X) = \bP_{\cC_2}^{ind}(G;X)$, hence they are d.p.-equivalent.
If $\cC_1 =  \bar{\cC_2}$, we use Proposition \ref{prop:complement} together with 
Proposition \ref{th:dp}. %Proposition 2.1.
\\
(ii) follows from Proposition \ref{prop:complement}.
\end{proof}
\fi %revised

%----------------------XXX-----------
\subsection{Generalized chromatic polynomials}
Recall from the introduction the definition of
$\chi_{\cC}(G;k)$ as the number of colorings of $G$
with at most $k$ colors such that each color class induces a graph in $\cC$.
\begin{theorem}[J. Makowsky and B. Zilber, cf. \cite{ar:KotekMakowskyZilber11}]
$\chi_{\cC}(G;k)$ is a polynomial in $k$
for any graph property $\cC$.
\end{theorem}

\ifrevised
\else
\begin{theorem}
\label{th:char-2}
Let $\cC_1$ and $\cC_2$ two d.p.-equivalent graph properties.
Then
$\chi_{\cC_1}(G;k)$ and $\chi_{\cC_2}(G;k)$
are d.p.-equivalent. 
\end{theorem}
\begin{proof}
This follows from the fact that
$G \in \cC$ iff $\chi_{\cC}(G;1) =1$.
\end{proof}
\fi %revised

In contrast to Proposition \ref{prop:complement}
the relationship between $\chi_{\cC}(G;k)$ and $\chi_{\bar{\cC}}(G;k)$
is not  at all obvious.

\begin{problem}
What can we say about $\chi_{\bar{\cC}}(G;k)$ in terms of $\chi_{\cC}(G;k)$?
\end{problem}

\begin{proposition}
\label{prop:char-2}
There are two classes $\cC_1$ and $\cC_2$ which are d.p.-equivalent but such that
$\chi_{\cC_1}$ and $\chi_{\cC_2}$ are not d.p.-equivalent.
\end{proposition}
\begin{proof}
Let $\cC_1$ be all the disconnected graphs and
Let $\cC_2$ be all the connected graphs. As they are complements of each other, they are d.p.-equivalent.
\\
We compute for $K_i$:
$$
\chi_{\cC_1}(K_i; j) =0, j \in \N^+
$$
because there is no way to partition $K_i$ into any number of disconnected  parts.
Hence $\chi_{\cC_1}(K_i; X) =0$.
$$
\chi_{\cC_2}(K_i; 2) = 2^i -2
$$
because every partion of $K_i$ into two nonempty parts gives two connected graphs.
\\
Therefore $\chi_{\cC_2}$ distinguishes between cliques of different size,
whereas $\chi_{\cC_1}$ does not.
\end{proof}

We note, however, that the analogue of Proposition \ref{th:char-1} for generalized chromatic polynomials
remains open.

\subsection{d.p.-equivalence of graph polynomials}
The converse of Theorem
\ref{th:char-1}(\ref{th:char-1.1}) and (\ref{th:char-1.2}) is not true:

\begin{proposition}
\label{prop:counterexample}
There are graph properties $\cC_1$ and $\cC_2$ which are not d.p.-equivalent,
but such that
\begin{enumerate}[(i)]
\item
$\bP_{\cC_1}^{ind}(G;X)$ and $\bP_{\cC_2}^{ind}(G;X)$
are d.p.-equivalent.
\item
$\chi_{\cC_1}(G;X)$ and $\chi_{\cC_2}(G;X)$
are d.p.-equivalent.
\end{enumerate}
\end{proposition}
\begin{proof}
For (i)
Let $\cC_1 = \{K_1\}$ and $\cC_2 =\{K_2, E_2\}$
%Let $\cC_1 = \{K_1\}$ and $\cC_2 =\{K_1, K_2, E_2\}$
where $E_n$ is the graph on $n$ vertices and no edges.
\\
%For (i) we compute: 
We compute: 
\begin{gather}
\bP_{\cC_1}^{ind}(G;X) = n(G)\cdot X \notag \\
\bP_{\cC_2}^{ind}(G;X) =  {n(G) \choose 2} \cdot X^2 \notag
%\bP_{\cC_2}^{ind}(G;X) = n(G)\cdot X + {n(G) \choose 2} \cdot X^2 \notag
\end{gather}
%--------
\ifrevised
For (ii) we
choose $\cC_1=\{K_1\}$ as before, but $\cC_2=\{K_1,K_2,E_2\}$. 
\\
{\bf Claim 1:} 
$\chi_{\cC_2}(G,X)\leq_{d.p.} n(G)$
\\§
Proof of Claim 1: Let $G_1$ and $G_2$ be two graphs with the same number of vertices. 
W.l.o.g. assume they have the same vertex set $V(G_1)=V(G_2)=V$.
Now notice for every $f:V \to [k]$, $f$ is a $\cC_2$-coloring of $G_1$ iff it is a 
$f$ is a $\cC_2$-coloring of $G_2$.
Hence $\chi_{\cC_2}(G_1,X)=\chi_{\cC_2}(G_2,X)$ whenever $G_1$ and $G_2$ have the same number of vertices. 
\\
{\bf Claim 2:}  $n(G) \leq_{d.p.} \chi_{\cC_2}(G,X)$
\\
Proof of Claim 2: First denote for every m,
$n_{even}(m) = \prod_{i=0}^{m-1} \binom{2(m-i)}{2}$
and 
$n_{odd}(m) = \prod_{i=0}^{m-1} \binom{2(m-i)+1}{2}$.
For every graph $G$, there is a natural number $m(G)$ such that $n(G)=2m(G)$ or $n(G)=2m(G)+1$.
If $n(G)=2m(G)$, $\chi_{\cC_2}(G,m(G))=n_{even}(m(G))$.
If $n(G)=2m(G)+1$, $\chi_{\cC_2}(G,m(G))=n_{odd}(m(G))$.  
Note $n_{odd}(r) > n_{even}(r)$ for every natural number $r$.
The minimal natural number $r$ such that $\chi(G,r)>0$ is equal to $m(G)$. 
We get that the minimal $r$ such that $\chi_{\cC_2}(G,r)>0$ determines $n(G)$. 
Hence $\chi_{\cC_1}$ and  $\chi_{\cC_2}$ are d.p.-equivalent. 
%--------------
\else
For (ii) we compute: 
\begin{gather}
\chi_{\cC_1}(G;X) =  {X \choose n}\cdot n! \notag \\
%\chi_{\cC_1}(G;X) = X^{n(G)} \notag \\
\chi_{\cC_2}(G;X) = 
\begin{cases}
\prod_{i=0}^m {2(m -i) \choose 2} \cdot {X \choose m} n!, & n(G)=2m  \\
%\prod_{i=0}^m {2(m -i) \choose 2}, & n(G)=2m  \\
0, & n(G) = 2m+1
\end{cases}
\notag 
\end{gather}
$\bP_{\cC_1}^{ind}(G;X)$ and $\bP_{\cC_2}^{ind}(G;X)$,
and $\chi_{\cC_1}(G;X)$ and $\chi_{\cC_2}(G;X)$,
are d.p.-equivalent, 
because  for both
two graphs $G_1, G_2$ get pairwise the same value of the polynomials iff $n(G_1)=n(G_2)$.
\fi %revised
%-------------------
\end{proof}
We leave it to the reader to construct the corresponding counterexample for
$\bP_{\cD}^{span}(G;X)$.

We cannot use Proposition \ref{prop:dp-incomparable} to show that there
infinitely many d.p.-incomparable graph polynomials of the form
$\bP_{\cC}^{ind}(G;X)$. However, we can construct explicitly 
infinitely many d.p.-incomparable graph polynomials of this form.

%-----------------------------------
%modified version for similar graphs
\subsection{Many d.p.-inequivalent graph polynomials}

For the rest of this section,
let 
$C_i$ be the undirected circle on $i$ vertices, and $C_i^*$ the graph which consists of a copy of $C_{i-1}$
together with a new vertex $v$ which is connected to exactly one of the vertices of $C_{i-1}$.
Clearly, $C_i$ and $C_i^*$ are similar.
Furthermore, let
$\cC_i = \{C_i\}$, and let $G_i^k$ consist of the disjoint union of $k$-many copies of $C_i$,
and let $\hat{G}_i^k$ consist of the disjoint union of $k-1$ copies of $C_i^*$ together
with one copy of $C_i$.
Again, $\hat{G}_i^k$ and $G_i^k$ are similar.

We compute:
\begin{lemma}
\label{incomp}
\begin{gather} 
\bP_{\cC_{j}}^{ind}(G_{i}^{k};X) = \bP_{\cC_{j}}^{ind}(\hat{G}_{i}^{k};X) = 0 \mbox{ for } i \neq j, i \neq j+1, \tag{i} \\ 
\bP_{\cC_{i}}^{ind}(G_{i}^{k};X) = k \cdot  X^{i} \tag{ii}  \\
\bP_{\cC_{i}}^{ind}(\hat{G}_{i}^{k};X) =  X^{i} \tag{iii}  
\end{gather}
\end{lemma}

\begin{theorem}
\label{th:dp-incomparable}
For all $i,j$ with $i \neq j$ and $i \neq j+1$ the polynomials 
$\bP_{\cC_{i}}^{ind}$ and
$\bP_{\cC_{j}}^{ind}$ are d.p.-incomparable, hence
there are infinitely many d.p.-inequivalent graph polynomials of the form $\bP_{\cC}^{ind}(G;X)$.
\end{theorem}

%------------------------done till here--------------------
\begin{proof}
Assume $i,j \geq 3$ with $i \neq j$ and $i \neq j+1$.
We first prove 
$\bP_{\cC_{i}}^{ind} \not <_{d.p.} \bP_{\cC_{j}}^{ind}$ for $i \neq j$ and $i \neq j+1$.
\\
We look at the graphs $G_j^2$ and $\hat{G}_j^2$.
$\bP_{\cC_{j}}^{ind}(G_{j}^{2};X) = 2 \cdot X^i$ by Lemma \ref{incomp}(ii).
$\bP_{\cC_{j}}^{ind}(\hat{G}_{j}^{2};X) = X^{i}$ by Lemma \ref{incomp}(iii).
Hence,
$\bP_{\cC_{j}}^{ind}$ distinguishes between the two graphs $G_j^2$ and $\hat{G}_j^2$.
However,
$\bP_{\cC_{i}}^{ind}(G_{j}^{2};X) = 
\bP_{\cC_{i}}^{ind}(\hat{G}_{j}^{2};X) = 0$, by Lemma \ref{incomp}(i).
Hence,
$\bP_{\cC_{i}}^{ind}$ does not distinguish between the two graphs.

To prove
$\bP_{\cC_{j}}^{ind} \not <_{d.p.} \bP_{\cC_{i}}^{ind}$ for $j \neq i$ and $j \neq i+1$.
we look at the graphs $G_i^2$ and $\hat{G}_i^2$.
In this case 
$\bP_{\cC_{j}}^{ind}$ does not distinguish between the two graphs $G_i^2$ and $\hat{G}_i^2$,
but $\bP_{\cC_{i}}^{ind}$ does.
\end{proof}

\begin{theorem}
\label{th:dp-incomparable-span}
There are infinitely many d.p.-inequivalent graph polynomials of the form $\bP_{\cC}^{span}(G;X)$.
\end{theorem}
\begin{proof}
The proof mimics the proof of Theorem \ref{th:dp-incomparable} with following changes:
Instead of $\cC_i$ we use $\cD_i = \{ C_i \sqcup E_j: j \in \N \}$ 
and
\begin{gather}
\bP_{\cC_{j}}^{span}(G_{i}^{k};X) = 0 \mbox{ for } i \neq j, i \neq j+1 \notag \\
\bP_{\cC_{i}}^{span}(G_{i}^{k};X) = k \cdot X^{i}. \notag \\
%-------------------------------------
\bP_{\cC_{j}}^{span}(\hat{G}_{i}^{k};X) = 
\begin{cases}
0 &  i \neq j, i \neq j+1 \\
(k-1) \cdot X^j & i=j+1
\end{cases}
\notag \\
\bP_{\cC_{i}}^{span}(\hat{G}_{i}^{k};X) =  X^{i}. \notag 
\end{gather}
\end{proof}

Next we look at chromatic polynomials $\chi_i(G;X) = \chi_{\cC_i}(G;X)$.
We use the following obvious lemma:
\begin{lemma}
\label{le:chrom}
\begin{enumerate}[(i)]
\item
For $X = \lambda \in \N$:
$$
\chi_i(G_i^k;\lambda) =
\begin{cases}
\lambda_{(k)} &  \lambda \geq k \\
0 & \mbox{ else }
\end{cases}
$$
\item
$$
\chi_j(G_i^k, \lambda) = 0 
$$
provided that $i \neq j$.
\item
$$
\chi_j(\hat{G}_i^k, \lambda) = 0 
$$
provided that $k \geq 2$ or  $k=1, i \neq j$.
\end{enumerate}
\end{lemma}

\begin{theorem}
\label{th:dp-incomparable-chrom}
For all $i \neq j$ the polynomials 
$\chi_i$ and $\chi_j$
are d.p.-incomparable, hence
there are infinitely many d.p.-incomparable graph polynomials of the form $\chi_{\cC}$.
\end{theorem}

\begin{proof}
$\chi_i \not \leq_{d.p.} \chi_j$:
\\
We look at the graphs $G_i^2$ and $\hat{G}_i^2$.
By Lemma \ref{le:chrom} $\chi_j$ does not distinguish between $G_i^2$ and $\hat{G}_i^2$.
However, $\chi_i$ distinguishes between them.

To show that
$\chi_j \not \leq_{d.p.} \chi_i$,
we look at the graphs $G_j^2$ and $\hat{G}_j^2$.
By Lemma \ref{le:chrom} $\chi_i$ does not distinguish between $G_j^1$ and $G_j^2$.
However, $\chi_j$ does distinguish between them.
\end{proof}

%-----------------------------------

\subsection{Generating functions of a relation}
\label{se:relation}
If, instead of counting induced (spanning) subgraphs with a certain graph property $\cC$ ($\cD$),
we count $r$-ary relations with a property $\Phi(A)$, we get a generalization of both
the generating functions of induced (spanning) subgraphs.
Here the summation is
defined by
$$
\bP_{\Phi}(G;X) = \sum_{A \subseteq E(G): \Phi(A)} X^{|A|}.
$$

For example,
the generating matching polynomial, defined as
$$
m(G;X) = \sum_{A \subseteq E(G): \Phi_{match}(A)} X^{|A|}.
$$
can be written as
$$
m(G;X) = \sum_{A \subseteq E(G): G\langle A \rangle \in \cD_{match}} X^{|A|}
$$
with  $\cD_{match}$ being the disjoint union of isolated vertices and isolated edges.

However, not every graph polynomial
$\bP_{\Phi}(G;X)$
can be written as a generating function of induced (spanning) subgraphs.

Consider the graph polynomial
$$
DOM(G;X)  = \sum_{A \subseteq V(G): \Phi_{dom}(A)} X^{|A|}
$$ 
where $\Phi_{dom}(A)$ says that $A$ is a dominating set of $G$.

We compute:
\begin{gather}
DOM(K_2,;X) = 2X +X^2  \label{dom1}\\
DOM(E_2,;X) = X^2 \label{dom2}
\end{gather}

\begin{theorem}
\label{th:dominating}
\begin{enumerate}[(i)]
\item
There is no graph property $\cC$ such that
$$
DOM(G;X) = \bP_{\cC}^{ind}(G;X).
$$
\item
There is no graph property $\cD$ such that
$$
DOM(G;X) = \bP_{\cD}^{span}(G;X).
$$
\end{enumerate}
\end{theorem}
\begin{proof}
(i):
Assume, for contradiction, there is such a $\cC$, and that
$K_1 \in \cC$.
The coefficient of $X$ in 
$\bP_{\cC}^{ind}(E_2;X)$ is $2$ because $K_1 \in \cC$. 
However, the coefficient of $X$ in $DOM(E_2;X)$ is $0$, by equation
(\ref{dom2}), a contradiction.

Now, assume $K_1 \not \in \cC$.
The coefficient of $X$ in 
$\bP_{\cC}^{ind}(K_2;X)$ is $0$, because $K_1 \not \in \cC$.
However, the coefficient of $X$ in $DOM(K_2;X)$ is  $2$, by equation
(\ref{dom1}), another contradiction.

(ii):
Assume, for contradiction, there is such a $\cD$.
The coefficient of $X$ in 
$\bP_{\cD}^{span}(K_2;X)$ is  $\leq 1$, because $K_2$ has only one edge. 
However, the coefficient of $X$ in $DOM(K_2;X)$ is  $2$, by equation
(\ref{dom1}), a contradiction.
\end{proof}

We can use Equation (\ref{dom1}) also to show the following:

\begin{theorem}
\label{th:notchrom}
There is no graph property $\cC$ such that
$$
DOM(G;X) = \chi_{\cC}(G;X).
$$
\end{theorem}

\begin{proof}
First we note that $\chi_{\cC}(G;1)=1$ iff $\chi_{\cC}(G;1) \neq 0$ iff $G \in \cC$.
\\
Assume that $K_2 \in \cC$.  Then we have, using Equation (\ref{dom1}),
$$
\chi_{\cC}(K_2;1) = 1 = DOM(K_2,1) =3,
$$
a contradiction.
\\
Assume that $K_2  \not \in \cC$.  then we have, using Equation (\ref{dom1}),
$$
\chi_{\cC}(K_2;1) = 0 = DOM(K_2,1) =3,
$$
another contradiction.
\end{proof}

\subsection{Determinant polynomials}

There are only two matrices associated with graphs which have been used to define
graph polynomials: the adjacency matrix and the Laplacian.
The two resulting determinant polynomials are d.p.-incomparable.
It is conceivable to to define other matrix presentations of graphs,
and ask when they give rise to d.p.-equivalent determinant polynomials. 
The characterization and recognition problem in this case amounts to the question
when the characteristic polynomial of a matrix is the
the characteristic polynomial arising from a graph.
However, in this paper we do not pursue this further.

\subsection{Characterizing d.p.-equivalence for special classes of graph polynomials}
Theorems \ref{th:char-1} and Proposition \ref{prop:char-2} and Proposition \ref{prop:counterexample}
show that d.p.-equivalence of $\cC$ and $\cC_1$, respectively $\cD$ and $\cD_1$, is not enough
to characterize d.p.-equivalence of generating functions or generalized chromatic polynomials
defined by $\cC$ and $\cD$.
Sometimes d.p.-equivalence of graph properties only implies and s.d.p.-equivalence 
of the corresponding graph polynomials.

\begin{problem}
Characterize d.p.-equivalence of graph polynomials arising from
$\cC$ and $\cD$ as
\begin{enumerate}
\item
generalized chromatic polynomials;
\item
generating functions of induced are spanning subgraphs;
\item
generating functions of relations.
\end{enumerate}
\end{problem}

\ifskip\else
\marginpar{Should we add something?}
\begin{framed}

Should we add something like this:

\begin{theorem}[Kind of....]
Given two graph polynomials 
$\bP_1(G;X)$ and 
$\bP_2(G;X)$, and a family of graphs
$G_n$ such that on $(G_n)_{n \in \N}$ 
$\bP_1(G;X)$ and 
$\bP_2(G;X)$ are d.p.-equivalent, \
and the sequence $\bP_1(G_n;X)$
is C-finite.
then there is a graph polynomial $\bP_3$ d.p.-equivalent to  $\bP_2(G;X)$
and the sequence $\bP_3(G_n;X)$
is C-finite.
\end{theorem}

There is a bit of freedom what we have to require on the special form of
the three graph polynomials $\bP_i(G;X)$ for $i=1,2,3$.
\end{framed}
\fi %skip

%\input{FPSAC-2016/fpsac-linalg}
%-----------------------------------------------------------------------
\ifsubmit\else
\begin{framed}
\noindent
File: NREV-conclu
\\
Label:  se:conclu
\end{framed}
\fi %submit
\section{Conclusions and open problems}
In the light of our general framework to study Wilf's characterization and recognition problems
for graph graph polynomials,
we have shown how to characterize the instances of  a graph polynomial $\bP(G;\bX)$
of various indexed sequences of graphs $G_{i}$ or $G_{i,j}$ using C-finite 
sequences of polynomials on $\Z[\bX]$.
Our method works for many graph polynomials and indexed sequences of graphs
as described in the general framework of
\cite{ar:FischerMakowsky08}, provided that each graph in the indexed sequence
is $\bP$-unique.
This improves the characterization of the instances of the defect matching polynomial
given in \cite{ar:GodsilGutman1981}, and generalizes it to infinitely many other graph polynomials
and indexed sequences of graphs.
It also shows that for graph polynomials given as generating functions of induced or spanning
subgraphs with a given property, C-finiteness is a semantic property.
It remains unclear, whether this also applies to generalized chromatic polynomials.

However, this approach to the algebraic characterization of graph properties,
as envisaged by the late Herbert Wilf in \cite{ar:Wilf1973}, is just a very small step
forward. The characterization of the polynomials which are instances of
the prominent graph polynomials,
the matching, chromatic or characteristic polynomials, remains wide open.

In the final section
we also briefly discussed whether and when such a characterization found for a graph polynomial
$\bP$ sheds light on other graph polynomials which are d.p.- or s.d.p-equivalent to $\bP$.
In forthcoming paper we shall discuss d.p.- or s.d.p-equivalence of graph polynomials
from a logical point of view, \cite{ar:KotekMakowskyRavve-APAL}.

\subsection*{Acknowledgements}

We would like to thank 
I. Averbouch and V. Rakita of our research seminar for their useful comments
during presentation of this material,
and
the referees of an earlier version of the paper, for their constructive comments.

%\input{FPSAC-2016/fpsac-conclu}
%---------------------------------------------------------------------------
\small
%\bibliography{sample}
%\input{bibfile-p}
\section*{References}
\newcommand{\etalchar}[1]{$^{#1}$}

%--------------------------------
%\input{FPSAC-2016/fpsac-ref-v1}
%-----------------------
%\input{REV-MaRa.bbl}
%
%\label{sec:biblio}
\end{document}

%END
\end{document}